\documentclass[10pt]{amsart}
\usepackage[utf8]{inputenc} 
\usepackage[T2A]{fontenc}
\usepackage{amsmath,calligra,mathrsfs}
\usepackage{graphicx}
\usepackage{amsthm, amsmath, amssymb,epic,eepic}
\usepackage{mathtools}
\usepackage{hyperref}

\theoremstyle{plain}
\newtheorem{theorem}{Theorem}[section]
\newtheorem{definition}[theorem]{Definition}
\newtheorem{lemma}[theorem]{Lemma}
\newtheorem{corollary}[theorem]{Corollary}
\newtheorem{proposition}[theorem]{Proposition}
\newtheorem{note}[theorem]{Note}

\newtheorem*{plan}{Plan of proof}
\newcommand{\Gm}{\mathbb{G}_m}

\usepackage{tikz-cd}
\usetikzlibrary{babel}

\newcommand{\bb}{\mathbb}

\newcommand{\cob}{\operatorname{cob}}

\newcommand{\Forget}{\operatorname{Forget}}

\newcommand{\pt}{\operatorname{pt}}

\newcommand{\id}{\operatorname{id}}

\newcommand{\Tr}{\operatorname{Tr}}
\newcommand{\Hcob}{\operatorname{H^0-cob}}

\def\geq{\geqslant}

\begin{document}
\title{Cobordism-framed correspondences and the Milnor $K$-theory}
\author{Aleksei Tsybyshev }
\address{St. Petersburg Branch of Steklov Mathematical Institute, Fontanka
27, St. Petersburg, 191023, Russia}
\address{Chebyshev laboratory, Saint Petersburg state university
14th Line 29B, Vasilyevsky Island, St.Petersburg 199178, Russia}
\email{emperortsy@gmail.com}

\begin{abstract}
In this paper, we compute the $0$th cohomology group of a complex of groups of cobordism-framed correspondences. In the case of ordinary framed correspondences, an analogous computation has been completed by A. Neshitov in his paper "Framed correspondences and the Milnor---Witt $K$-theory". 

Neshitov's result is, at the same time, a computation of the homotopy groups $\pi_{i,i}(S^0)(Spec(k)),$ and the present work could be used in the future as basis for computing homotopy groups $\pi_{i,i}(MGL_{\bullet})(Spec(k))$ of the spectrum $MGL_{\bullet}.$
\end{abstract}
\maketitle{}

\section{Introduction}
 The theory of framed correspondences and framed transfers was conceived by V. Voevodsky in~\cite{Voev}. In the course of developing that theory, G. Garkusha and I. Panin in~\cite{GP} defined and studied framed motives of algebraic varieties. One application of that theory is an explicit fibrant replacement of the suspension bispectrum $\Sigma_{S^1}^{\infty}\Sigma_{\Gm}^{\infty}X_{+}$ of a smooth variety $X\in Sm_k$.

as a corollary, one can reduce the computation of motivic homotopy groups
$\pi^{n,n}(\Sigma_{S^1}^{\infty}\Sigma_{\Gm}^{\infty}S^0)(k)$ to the computation of the $0$th cohomology group
$H^0(\mathbb{Z}F(\Delta_k^{\bullet},\Gm^{\wedge n}))$ of an explicit simplicial abelian group $\mathbb{Z}F(\Delta^{\bullet}_k,\Gm^{\wedge
n})$ (see~\cite[следствие 10.7]{GP}). That computation was completed by Neshitov in his paper~\cite{Milnor-Witt}.

It appears that, using~\cite{GarNesh}, one can computa an analogous motivic homotopy group $\pi^{n,n}(MGL_{\bullet})(k)$ as the $0$th cohomology group of the complex $\mathbb{Z}F^{\cob}(\Delta_k^{\bullet},\Gm^{\wedge n}).$ The groups $\mathbb{Z}F^{\cob}$ are defined in Section~\ref{def}. The goal of the present paper is to calculate these $0$th cohomology groups, or, more precisely, to prove Theorem~\ref{main}, in which we present an explicit isomorphism of graded rings.
$$
\oplus \sigma_m\colon  \bigoplus K_m^M(k) \to \bigoplus H^0\big(C_* \mathbb{Z}F^{\cob}(\pt,\mathbb{G}_m^{\wedge m})\big).
$$

The author thanks professor Panin for presenting the problem to him, and for advice on the properties of Milnor $K$-groups on curves.

\section{Definition of cobordism-framed correspondences and statement of the main result}\label{def}

We begin with repeating the definitions of some basic objects.

\begin{definition}
Let $Z \subseteq X$ be a closed subscheme. An \'etale neighbourhood of $Z$ in $X$ is an \'etale morphism $e\colon  W \to X,$ such that $W \times_X Z \to Z$  is an isomorphism. 

In that context, sometimes the scheme $W$ itself, in that case the morphism $e$ is implicit.
\end{definition}

\begin{definition}
Let $e,e'$ be two \'etale neighbourhoods of $Z$ in $X.$ The neighbourhood $e'$ is called a \textbf{refinement} of $e$, if $e'$ factors through $e$, i.e. there exists a morphism $f$ (which is by necessity \'etale), such that $e'=e \circ f.$
\end{definition}

The following is a basic definition of the theory of framed correspondences, and was originally given by Voevodsky in~\cite[раздел 2]{Voev} as a ``globally framed correspondence''.

\begin{definition}
Let $X,Y$ be schemes. An explicit framed correspondence 
 of level $n$ is comprised of the following data:
\begin{itemize}
    \item {a closed subset $ Z \subseteq \mathbb{A}^n_X$, finite over $X$, called the \textbf{support} of the correspondence};
    \item {an \'etale neighbourhood $W \supset Z$ of $Z$ in $\mathbb{A}^n_X$};
    \item {a morphism of schemes $\phi \colon  W \to \mathbb{A}^n$,  such that the subset $Z \subseteq W$ is the preimage of $0$ under $\phi$};
    \item {a morphism $g\colon  W \to Y$.}
\end{itemize}

Such an explicit framed correspondence is denoted by a tuple $(Z,W,\phi,g)$.

Two explicit framed correspondences of the same level $n$ are called equivalent: 
$$
(Z,W,\phi,g) \sim (Z,W',\phi ',g'),
$$
 if they have the same support $Z$ and there exists a refinement $W''$ of both $W$ and $W'$ (implying morphisms $i\colon  W'' \to W$ and $i'\colon W'' \to W'$ over $\mathbb{A}^n_X$ ), such that

$$
\phi \circ i= \phi' \circ i',
$$

$$
 g \circ i = g' \circ i'.
$$

The set  $ Fr_n(X,Y) $ of level $n$ framed correspondences from $X$ to $Y$ is the set of equivalence classes with respect to that equivalence relation.
By abuse of notation, the class of the explicit framed correspondence $(Z,W,\phi,g)$ is also denoted by $(Z,W,\phi,g).$ Single letters, such as $c,$ are also used to denote framed correspondences.
\end{definition}

\begin{note}
The set $ Fr_n(X,Y) $ is naturally pointed with the correspondence $\varnothing,$ the only equivalence class of correspondences with empty support.
\end{note}

\begin{note}
Based on that definition Voevodsky in~\cite{Voev} defined pointed sets $Fr(X,Y)$ of stable framed correspondences.
On the same basis, in~\cite[Definitions 2.8, 8.3, 8.5]{GP} the groups of linear and stable linear framed correspondences 
$\bb{Z}F_n(X,Y)$, $\bb{Z}F(X,Y)$ cоответственно. 
All of these definitions are repeated in~\cite[Section 1]{Milnor-Witt}. 
These objects are used in the present paper.
\end{note}

We now give an analogous definition, central for the present paper.
\begin{definition}\label{cob_unstable}

Let $X$ be a smooth variety over the field $k$ of characteristic $0$, and let $Y$ be a presheaf on the category of smooth varieties over $k.$ 

 An \textbf{explicit cobordism-framed correspondence} of level $(n,N)$ is comprised of the following data:
\begin{itemize}
    \item {A closed subset $ Z \subseteq \mathbb{A}^n_X$ finite over $X$}
    \item {An \'etale neighbourgood $W \supset Z$ in $\mathbb{A}^n_X$ }
    \item {A regular map $\phi : W \to \tau_{n,N}$, where $\tau_{n,N}$ is the total space of the tautological bundle $\tau_{n,N}^{Sh}$ over the Grassmann variety $Gr_{n,N}$; such that the subset $Z \subseteq W$ is the preimage of the zero-section under $\phi$ }
    \item {A morphism $g: W \to Y,$ i.e. $g \in \Gamma(W,Y)$}
\end{itemize}

Such an explicit cobordism-framed correspondence is denoted by a tuple $(Z,W,\phi,g)$.

Two explicit cobordism-framed correspondences of the same level $(n,N)$ are called equivalent: 
$$
(Z,W,\phi,g) \sim (Z,W',\phi ',g'),
$$
 if they have the same support $Z$ and there exists a refinement $W''$ of both $W$ and $W'$ (implying morphisms $i\colon  W'' \to W$ and $i'\colon W'' \to W'$ over $\mathbb{A}^n_X$ ), such that

$$
\phi \circ i= \phi' \circ i',
$$

$$
 g \circ i = g' \circ i'.
$$

The set   $ Fr_{n,N}^{cob}(X,Y) $  of level $(n$ framed correspondences from $X$ to $Y$ is the set of equivalence classes with respect to that equivalence relation.
By abuse of notation, the class of the explicit framed correspondence $(Z,W,\phi,g)$ is also denoted by $(Z,W,\phi,g).$ Single letters, such as $c,$ are also used to denote framed correspondences.

The set $ Fr_{n,N}^{cob}(X,Y) $ of cobordism-framed correspondences of level $(n,N)$ from $X$ to $Y$  is the set of equivalence classes with respect to that equivalence relation.
By abuse of notation, the class of the explicit cobordism-framed correspondence $(Z,W,\phi,g)$ is also denoted by $(Z,W,\phi,g).$ Single letters, such as $c,$ are also used to denote cobordism-framed correspondences.
\end{definition}
\begin{note}
While regular maps $X \to Gr_{n,N}$ correspond to epimorphisms $\mathcal{O}_X^N \to E$, where $E$ is a locally free sheaf of rank $n$, regular maps $X \to \tau_{n,N}$ correspond to the same data (since there is a canonical morphism $\tau_{n,N} \to Gr_{n,N}$), plus a choice of a section $s \in \Gamma(X,E).$   
\end{note}
\begin{note}
The notation $(Z,W,\phi,g)$ is brief but implicit, because "$W \supset Z$ being an \'etale neighbourhood in $\mathbb{A}^n_X$" implies an \'etale morphism $W \to \mathbb{A}^n_X,$ and a closed embedding $Z \to W$ making the triangle commute:
\[
\begin{tikzcd}
Z \arrow[d,hook] \arrow[r,hook] & W \arrow[dl] \\
  \mathbb{A}^n_X
\end{tikzcd}.
\]
\end{note}
\begin{definition}
The set $ Fr^{cob}(X,Y) $ of stable cobordism-framed correspondences from $X$ to $Y$  is the limit of $ Fr_{n,n+N}^{cob}(X,Y) $ when $N \to \infty$ and $n \to \infty$ along the following maps:

    along N: $\tau_{n,n+N} \to \tau_{n,n+N+1}$ is given on the represented functors by the natural transformation $\left( \mathcal{O}^{n+N} \twoheadrightarrow E , s \right) \mapsto \left( \mathcal{O}^{n+N+1} \twoheadrightarrow E, s \right) $, where the resulting epimorphism is zero on the last coordinate.
		
    Along n: $(Z,W,\phi,g) \mapsto (Z',W',\phi',g'),$ where: 
    \begin{itemize}
        \item  {The subset $Z'$ is taken to be the image of $Z \hookrightarrow \mathbb{A}^n \xhookrightarrow{i_2} \mathbb{A}^1 \times \mathbb{A}^n \simeq \mathbb{A}^{1+n},$ where $i_2$ is the embedding as the second factor, equaling zero on the first coordinate.}
        \item {The neighbourhood $W'$ is $\mathbb{A}^1 \times W$.}
        \item {The map $\phi'$, supposing $\phi$ corresponds to $\mathcal{O}^{n+N} \twoheadrightarrow E$ and $s \in \Gamma(X,E)$, is given by $\mathcal{O}^{1+n+N} \simeq \mathcal{O} \oplus \mathcal{O}^{n+N} \twoheadrightarrow \mathcal{O} \oplus E$} and the section $(x_{-n-1},s)$, where $x_{-n-1}$ is the coordinate function $pr_1 : \mathbb{A}^1 \times W \to \mathbb{A}^1$
        \item {The morphism $g'$ is $\mathbb{A}^1 \times W \to W \to Y$}

        \end{itemize}
\end{definition}
\begin{note}
One has to check that the stabilization maps in this definition are well-defined: $Z$ is set-theoretically the preimage of the zero-section because one of the coordinate functions cuts out $W$ in $\mathbb{A}^1 \times W$, and the rest of them cut out $Z$ in $W$. Also, to simultaneously pass to the limit along $n$ and $N$, one has to check that the stabilization maps commute with each other. This is achieved by noticing that they add coordinates on different sides --- one on the left, the other on the right.
\end{note}
\begin{definition}
The group  $\mathbb{Z}F^{cob}(X,Y)$ of linear stable cobordism-framed correspondences from $X$ to $Y$ is the abelian group with generators $Fr^{cob}(X,Y)$ and relations 

\[\left[(Z,W,\phi,g)\right] +\left[(Z',W',\phi',g')\right] = \left[(Z \amalg Z', W \amalg W', \phi \amalg \phi', g \amalg g')\right].\]

 The representatives here are taken on some finite level, in one particular $Fr_{n,N}^{cob}(X,Y),$ and the relation only makes sense if $Z \cap Z' = \emptyset.$
\end{definition}

\begin{note}
$\mathbb{Z}F_{n,n+N}^{\cob}(X,Y)$ is also a free abelian group with basis comprised of the correspondences with connected support.
\end{note}

\begin{definition}
The group $\mathbb{Z}F^{\cob}(X,Y)$ of stable linear cobordism-framed correspondences from $X$ to $Y$ is the inductive limit of the groups $\mathbb{Z}F_{n,n+N}^{\cob}(X,Y)$ as $N \to \infty$ and $n \to \infty.$ 
\end{definition}

Taking $X$ in the definition above to be the members of a cosimplicial object $\Delta_k^{\bullet} \times X$, we get a complex $C_* \mathbb{Z}F^{\cob}( X, Y)$, which we denote by $C_* \mathbb{Z}F^{\cob}(X,Y).$

\begin{lemma}
Inside the complex $C_* \mathbb{Z}F^{\cob}(X,\Gm^{\{1 \dots m\}})$ the sum of the subcomplexes 
$$
C_* \mathbb{Z}F^{\cob}\big(X,\Gm^{\{1 \dots \widehat{i} \dots m\}}\big)
$$
 can be split out as a direct summand.
\end{lemma}

\begin{proof}
Denote the idempotent correspondence (which is actually a map) 
$$
{\Gm \to \pt \xrightarrow{1} \Gm}
$$
 as $e$. On $\Gm^{\times m}$ there are $m$ idempotent correspondences, one for each factor. Denote them as $e_1, \dots, e_m.$ Note that the $e_i$ commute pairwise, since the product of any collection $\{e_i,i \in I\}$ is a result of adding factors with the identity morphism to the morphism 
$$
\Gm^{I} \to \pt \xrightarrow{(1 \dots 1)} \Gm^{I}.
$$
 In particular, the products of the~$e_i$ are also idempotents. The same is true for $(1-e_i).$  The composition with these idempotents gives us corresponding idempotent maps on 
$$
C_* \mathbb{Z}F^{\cob}\big(X,\Gm^{\{1 \dots m\}}\big).
$$ 

The idempotent 
$$
1-\prod_{i=1}^m (1-e_i)
$$
 splits out as a direct summand the sum of subcomplexes $C_* \mathbb{Z}F^{\cob}(X,\Gm^{\{1 \dots \widehat{i} \dots m\}}),$ and, respectively, the idempotent 
$$
\prod_{i=1}^m (1-e_i)
$$
 splits out its complement.
\end{proof}

\begin{definition}
Following~\cite{SV} and~\cite{Milnor-Witt}, the direct complement in
$$
C_* \mathbb{Z}F^{\cob}\big(X,\Gm^{\{1 \dots m\}}\big)
$$
 of the sum of the subcomplexes 
$$
C_* \mathbb{Z}F^{\cob}\big(X,\Gm^{\{1 \dots \widehat{i} \dots m\}}\big)
$$
 (split out by the idempotent $\prod_{i=1}^m (1-e_i)$ from the proof above) is denoted by $C_* \mathbb{Z}F^{\cob}(X,\Gm^{\wedge m}).$
\end{definition}

 We consider the case when $X=\pt$ is one rational point in more detail: the main result of the paper is the identification in this case of the $0$th group of the complex above with the $m$th $K$\nobreakdash-group of the field $k.$  (Recall that Neshitov in~\cite{Milnor-Witt} identified the $0$th cohomology group of the complex $\bb ZF(\Delta^{\bullet}_k; \Gm^{\wedge m})$ with the Milnor-Witt $K$-group $K^{MW}_m(k)$.)

To state the claim more preciselym we introduce an external product structure on the groups $H^0(C_* \mathbb{Z}F^{\cob}(X,Y)).$ The construction is analogous to \cite[Section 3]{Milnor-Witt}, but it uses the direct summation map. 

\begin{definition} \label{d:external_product}
Let 
\begin{align*}
c&=(Z,W,\phi,g) \in Fr_{n,N}^{\cob}(X,Y),
\\
c'&=(Z',W',\phi',g') \in Fr_{n',N'}^{\cob}(X',Y').
\end{align*}
Then their external product is the class of the explicit cobordism-framed correspondence
$$
c \times\! c'\! =\!\big((Z \!\times\! Z'\!, W \!\times\! W'\!, i_{((n,N),(n'\!,N'))}) \circ (\phi \!\times\! \phi'),g \!\times\! g' \!\in\! Fr_{n\!+\!n',N\!+\!N'}^{\cob}(X \!\times\! X'\!,Y \!\times\! Y')\big),
$$
where the morphism 
$ i_{((n,N),(n',N'))})\colon  \tau_{n,N} \times \tau_{n',N'} \to \tau_{n+n',N+N'}$ is defined on the represented functors by taking the direct sum as follows:
$$
\begin{tikzcd}
\tau_{n,N} \times \tau_{n',N'} \arrow[d] 
&  \big(( p\colon  \mathcal{O}^{n+N} \twoheadrightarrow E , s \big), \big( p'\colon \mathcal{O}^{n'+N'} \twoheadrightarrow E' , s' ) \big) \arrow[d,mapsto] \\
\tau_{n+n',N+N'}
&  \big( p \oplus p' \circ T_{n,N,n',N'} \colon  \mathcal{O}^{n+n'+N+N'} \twoheadrightarrow E \oplus E', (s,s' )  \big).
\end{tikzcd}
$$
Here
$$
\begin{tikzcd}
T_{n,N,n',N'}\colon \\
 \mathcal{O}^{n+N+n'+N'} \arrow[d,"\simeq"] \\
\mathcal{O}^{n} \oplus \mathcal{O}^{N} \oplus \mathcal{O}^{n'} \oplus \mathcal{O}^{N'} \arrow[d,"\simeq"] \\ 
\mathcal{O}^{n} \oplus \mathcal{O}^{n'} \oplus \mathcal{O}^{N} \oplus \mathcal{O}^{N'} \arrow[d,"\simeq"] \\ 
\mathcal{O}^{n+n'+N+N'}\\
\end{tikzcd}
$$
swaps the places of the two direct summands.

\end{definition}

\begin{note}
This  multiplication operation does not respect the stabilisation maps, so it does not give rise to an external multiplication on stable cobordism-framed correspondences. However, it does respect the stabilisation maps up to homotopy, so the following lemma is true.
\end{note}

\begin{lemma}
The external multiplication of cobordism-framed correspondences gives rise to a well-defined external multiplication operation 
$$
H^0(C_* \mathbb{Z}F^{\cob}(X,Y)) \otimes H^0\big(C_* \mathbb{Z}F^{\cob}(X',Y')\big) \to H^0\big(C_* \mathbb{Z}F^{\cob}(X \times X',Y \times Y')\big).
$$

\end{lemma}

\begin{theorem} \label{main}
Sending a symbol $\{g_1, \dots, g_m\}$ to a level $0$ correspondence, specifically, a map 
$$
\widetilde{\sigma}_m(g_1, \dots, g_m) \colon  \pt \to \mathbb{G}_m^m,
$$
given by the coordinates $(g_1, \dots , g_m),$ and then taking the class of that correspondence in $H^0(C_* \mathbb{Z}F^{\cob}(\pt,\mathbb{G}_m^{\wedge m}))$ \textup(first according to the stabilisation\textup, then when going to the quotient $\mathbb{G}^{\wedge m}$ of $\mathbb{G}^m,$ and\textup, finally\textup, according to homotopy\textup)\textup, gives rise to a well-defined group homomorphism 
$$
\sigma_m\colon  K_m^M(k) \to H^0\big(C_* \mathbb{Z}F^{\cob}(\pt,\mathbb{G}_m^{\wedge m})\big).
$$
 Together these homomorphisms form a graded ring homomorphism
$$
\oplus \sigma_m\colon  \bigoplus K_m^M(k) \to \bigoplus H^0\big(C_* \mathbb{Z}F^{\cob}(\pt,\mathbb{G}_m^{\wedge m})\big).
$$
This homomorphism is an isomorphism.
\end{theorem}

\section{Comparison with the framed correspondence case}

In this section, we use some notation and definitions from the paper \cite{Milnor-Witt}.
 
Let us define maps, for all $N \geq n,$ 

\[cob_{n,N} : Fr_{n}(X,Y) \to Fr_{n,N}^{cob}(X,Y),\]

that are natural in $X$ and $Y$ (from the categories of varieties and presheaves respectively), and that, as a whole, commute with the stabilisation maps.

  Take the image of the correspondence $(Z,W,\phi, g)$ to be $(Z,W,i_{n,N} \circ \phi, g)$, where $i_{n,N}$ is defined as follows: The point $* \in Gr_{n,N}$ is given by the projection to the first $n$ coordinates $k^N \to k^n.$ So the fiber $\tau_{N,n}^{Sh}|_*$ of the tautological bundle on $Gr_{N,n}$ is canonically isomorphic to $k^n.$ Accordingly, the fiber of the total space is $\mathbb{A}^n \simeq \tau_{n,N} \times_{Gr_{n,N}} *$ canonically. Taking the composition of this isonorphism with the embedding of the fiber, we get
	
	\[i_{n,N}: \mathbb{A}^n \simeq \tau_{n,N} \times_{Gr_{n,N}} * \hookrightarrow \tau_{n,N}.\] 
	
A simple check shows that the maps $cob_{n,N}$ commute with stabilisation maps, and that they preserve the extra additivity relations for $\mathbb{Z}F $ and $\mathbb{Z}F^{cob}$, which lets us define the natural maps

\[cob : Fr(X,Y) \to Fr^{cob}(X,Y)\]

 and natural homomorphisms

\[\mathbb{Z}cob : \mathbb{Z}F(X,Y) \to \mathbb{Z}F^{cob}(X,Y). \]

Thus, there is a homomorphism

\[\widetilde{\Hcob}_{X,Y} : H^0(C_* \mathbb{Z}F( X, Y)) \to H^0(C_* \mathbb{Z}F^{cob}(X, Y)),\]

 and, for each $m,$ the principal direct summand in $\widetilde{\Hcob}_{pt,\mathbb{G}_m^{\times m}}$:

 \[\Hcob : H^0(C_* \mathbb{Z}F(pt, \mathbb{G}_m^{\wedge m})) \to H^0(C_* \mathbb{Z}F^{cob}(pt, \mathbb{G}_m^{\wedge m})).\]

 We will soon show that this homomorphism is onto.

  We begin by naturally parameterising the data $Fr_{n,N}^{cob}(X,Y),$ that have the bundle $E$ be trivial, by some data including a choice of the trivialisation ; we explore the properties of this parameterisation.

\begin{definition}
The set $Fr_{n,N}^{cob,triv}(X,Y)$ of \textbf{trivialised cobordism-framed correspondences} consists of classes of tuples $(Z,W,A,v,g)$, where $Z,W$ and $g$ are the same as in Definition \ref{cob_unstable}, and the equivalence relation is defined similarly (by refining $W$), but instead of $\phi$ the following "numerical" datum is used: A matrix $A$ $N \times n$ and rank $n$ (at all points) and a vector $v$ --- a column of height $n.$ The entries of $A$ and $v$ are in the ring $k[W].$ It is required that $\{v=0\} = Z.$ 
\end{definition}
That datum gives rize to a regular map $\phi_{A,v} : W \to \tau_{n,N},$ taking for the bundle $E$ the trivial bundle $\mathcal{O}^n$, and for the epimorphism, the linear operator $\mathcal{O}^N \to \mathcal{O}^n$ given by the matrix $A,$ which is an epimorphism, by the rank requirement; the section  $s$ is given by the vector $v$. In fact, choosing a preimage of $\phi = (\mathcal{O}^N \twoheadrightarrow E, s)$ under the map $(A,v) \to \phi_{A,v}$ is equivalent to choosing a trivialisation of $E.$ Applied to $Fr_{n,N}^{cob,triv}(pt,Y),$ $W$ is an arbitrarily small neighbourhood of a finite number of points. Any bundle $E$ is trivial on a small enough $W,$ so
\begin{lemma} \label{all_triv}
If $X = pt,$ then the map 

\[Forget_{n,N} : Fr_{n,N}^{cob,triv}(X,Y) \to Fr_{n,N}^{cob}(X,Y)\]

 is onto.
\end{lemma}
 For $N>n$ Gaussian elimination for columns reduces $A$ to the matrix of projection onto the first  $n$ coordinates, using elementary operations of type 1. To each elementary operation $t_{i,j}(\lambda)$ corresponds an elementary operation $t_{i,j}(\lambda x)$ over $k[W][x].$ Consider the element

\[\left(Z \times \mathbb{A}^1, W \times \mathbb{A}^1,  A \cdot t_{i,j}(\lambda x),v,g \right) \in Fr_{n,N}^{cob,triv}(\mathbb{A}^1,Y).\]

It gives a homotopy between $(Z,W,A,v,g)$ 

and $(Z,W,A\cdot t_{i,j}(\lambda),v,g).$

 From this and the previous note, it follows that

\begin{equation} \label{Gauss}
[(Z,W,A,v,g)] = [(Z,W,P,v,g)] \in \pi_0 \left( Fr_{n,N}^{cob,triv}(X \times \Delta_k^{\bullet},Y) \right),
\end{equation}
 where $P$ is the matrix of projection onto the first $n$ coordinates.

Note that, denoting the map $W \to \mathbb{A}^n$ corresponding to $v$ by universal property by $a_v,$

\[\phi_{P,v} = i_{n,N} \circ a_v.\]

 Thus
\begin{equation} \label{cob-forget}
Forget_{n,N} \left(Z,W,A,v,g\right) = cob_{n,N} (Z,W,a_v,g) 
\end{equation}
\begin{corollary} \label{H0-cob}
The homomorphism $\widetilde{\Hcob}_{X,Y}$ is onto if $X=pt$. Hence the homomorphism $\Hcob$ is onto, as it is a direct summand.
\end{corollary}
\begin{proof}
The maps $Forget_{n,N}$ are onto by Lemma \ref{all_triv}. Suppose the element in question can be written as

\[Forget_{n,N}(Z,W,A,v,g).\]

 Equality \ref{Gauss} shows that it is homotopic to a "trivial" correspondence

\[Forget_{n,N}(Z,W,P,v,g).\]

 Equality \ref{cob-forget} shows that this is in the image of $cob_{n,N}.$ 
\end{proof}

Any change of basis in $E$ provides another parameter with the same image under $Forget_{n,N}.$ Hence the equality:

 \begin{equation} \label{GL} 
\forall M \in GL_n(k[W]),  \Forget_{n,N}(Z,W,A,v,g) =\Forget_{n,N} (Z,W,M A,M v,g). 
\end{equation} 

The following lemma can be interpreted as the statement that for $X=pt,$ the class of a correspondence in $H^0(\mathbb{Z}F^{cob}(pt,Y))$ is independent on the differential map of the framing map $\phi$. 

\begin{lemma} \label{NoDiff}
For $X=pt,$

$$\forall M \in GL_n(k),  cob_{n,N}(Z,W,a_v,g) \sim cob_{n,N}(Z,W,a_{M v},g),$$

i.e. their classes in $ H^0(\mathbb{Z}F^{cob}(X, Y)$ are equal.
\end{lemma}
\begin{proof}
\begin{equation*} 
\begin{split}
\cob_{n,N}(Z,W,a_v,g)&\stackrel{\mbox{{\tiny equality \ref{cob-forget}}}}{=}
\Forget_{n,N}(Z,W,P,v,g) 
\\
&\stackrel{\mbox{{\tiny equality \ref{GL}}}}{=} \Forget_{n,N}(Z,W,M P,M v,g)
\\
  &\stackrel{\mbox{{\tiny equality \ref{Gauss}}}}{\sim}
\Forget_{n,N}(Z,W,P,M v,g) 
\\
&\stackrel{\mbox{{\tiny equality \ref{cob-forget}}}}{=}
\cob_{n,N}(Z,W,a_{M v},g). \qedhere
\end{split}
\end{equation*}
\end{proof}

This allows us to reduce various correspondences to maps:

\begin{proposition} \label{reduce_to_value}
If $X=pt,$ and $Y$ is an open subvariety in an affine space, then any cobordism-framed correspondence $Z,W,\phi,g$ with $Z \simeq Spec(k)$ a single rational point, cut out transversally by the zero-section, has the same class  in $H^0(C_* \mathbb{Z}F^{cob}(pt,Y))$ as the map (correspondence of level $0$) $g|_Z$  \end{proposition}
\begin{proof}
By Corollary \ref{H0-cob}, the class of the correspondence in question is the image of some framed correspondence with the same $Z$ and $g|_Z$ under the map $\widetilde{\Hcob}_{X,Y}$

By \cite{Milnor-Witt}[4.10] (Where, in fact, the correspondence is proven to be equivalent to a single correspondence of level $1$, and not their sum, and the map $g|_Z$ stays the same), we reduce to the case of a level $1$ correspondence. By \cite{Milnor-Witt}[Lemma 5.2], this correspondence is equivalent, as a cobordism-framed correspondence, to

\[\mathbb{A}^1_k, \{\mu(t-\lambda)=0\},i_{1,1} \circ a_{\mu(t-\lambda)},g(\{\mu(t-\lambda)=0\}).\]

 By Lemma \ref{NoDiff}, $\mu$ can be homotopied to $1.$ Then it is easy to present a (translation) homotopy taking $\lambda$ to $0.$ As a result, we get exactly the image of thelevel $0$ correspondence

\[(Spec(k), Spec(k), id, g(Z))\]

 under the $n$-wise stabilisation map.
\end{proof}

 The natural transformation $\widetilde{\Hcob}_{-,-}$ respects the external multiplication operations on framed correspondences and cobordism-framed correspondences.

\begin{lemma} \label{l:H0-cob_external_multiplication}
There are external multiplication operations on the groups $H^0(C_*\bb ZF(-,-))$ (see~\cite[Section 3]{Milnor-Witt}) and $H^0(C_*\bb ZF^{cob}(-,-))$ (see~\ref{d:external_product}).
Those are compatible with the natural transformation $\widetilde{\Hcob}_{-,-},$ which means that the following square commutes:
\[
\begin{tikzcd}
H^0(C_* \mathbb{Z}F(X, Y)) \otimes H^0(C_* \mathbb{Z}F(X', Y')) \arrow[swap,d,"\widetilde{\Hcob}_{X,Y} \otimes \widetilde{\Hcob}_{X',Y'}"] \arrow[r,"m_F"] & H^0(C_* \mathbb{Z}F(X \times X', Y \times Y')) \arrow[d,"\widetilde{\Hcob}_{X \times X', Y \times Y'}"]\\
H^0(C_* \mathbb{Z}F^{cob}(X, Y)) \otimes H^0(C_* \mathbb{Z}F^{cob}( X', Y')) \arrow[swap,r,"m_{F^{cob}}"] & H^0(C_* \mathbb{Z}F^{cob}(X \times X', Y \times Y'))
\end{tikzcd}.
\]
\end{lemma}

\begin{proof}
The statement follows from the commutative square below:

\[\begin{tikzcd}
\mathbb{A}^n \times \mathbb{A}^{n'} \arrow[d] \arrow[rr, "i_{n,N} \times i_{n',N'}"] && \tau_{n,N} \times \tau_{n',N'} \arrow[d] \\
\mathbb{A}^{n+n'} \arrow[rr,"i_{n+n',N+N'}"] && \tau_{n+n',N+N'}
\end{tikzcd}\]
\end{proof}

\begin{corollary} \label{cob-gr}
The external multiplication defined above provides a structure of a graded ring with a unit on
\[\bigoplus \limits_{m \geq 0} H^0(\mathbb{Z}F^{cob}(pt,\mathbb{G}_m^{\wedge m}) ),\]

 and the maps $\Hcob$ provide a homomorphism of graded rings with a unit

\[
\bigoplus \limits_{m \geq 0} H^0(\mathbb{Z}F(pt,\mathbb{G}_m^{\wedge m}) ) \to \bigoplus \limits_{m \geq 0} H^0(\mathbb{Z}F^{cob}(pt,\mathbb{G}_m^{\wedge m}) )
\]
\end{corollary}

\begin{proof}
The unit is the identity map $pt \to pt.$ It goes to itself under $cob.$
\end{proof}

\section{Proof of the main result} \label{proofMain}

\begin{plan}
In this section we will show that the $\sigma_m$ of Theorem~\ref{main} are isomorphisms

\[\sigma_m: K^M_m(k) \simeq H^0(C_* \mathbb{Z}F^{cob}(pt,\mathbb{G}_m^{\wedge m})).\]

 The proof follows the same general plan as Voevodsky's proof for the correspondences $Cor,$ and Neshitov's proof for correspondences $\mathbb{Z}F.$ 

First we check in Proposition~\ref{sigma} that the maps introduced in the statement of Theorem~\ref{main} are well-defined and provide a graded ring homomorphism.

After that, in Proposition~\ref{surj} we show that the maps $\sigma_m$ are onto.

Then, for each $m$, in Proposition \ref{rho_m} we give a well-defined map
$$
\rho_m \colon H^0\big(C_* \mathbb{Z}F^{\cob}(\pt,\mathbb{G}_m^{\wedge m})\big) \to K^M_m(k).
$$

These maps go in the opposite direction of $\sigma_m$:

$$
\sigma_m\colon  K^M_m(k) \to H^0\big(C_* \mathbb{Z}F^{\cob}(\pt,\mathbb{G}_m^{\wedge m})\big),
$$

 and $\rho_m$ is a candidate for the inverse map to $\sigma_m.$ It is easy to see from the definitions that $\rho_m \circ \sigma_m = \id.$ The same is not clear for the other composition $\sigma_m \circ \rho_m,$ however, $\sigma_m$ being onto, these two maps turn out to be mutually inverse isomorphisms.

\end{plan}

\begin{proposition} \label{sigma}
Taking the symbol $\{g_1, \cdots, g_m\}$ to the level $0$ correspondence, specifically a map, 

\[\tilde{\sigma}_m(g_1, \cdots, g_m) : pt \to \mathbb{G}_m^m,\] 

given by the coordinates $(g_1, \cdots , g_m),$ and taking that to its class in $H^0(C_* \mathbb{Z}F^{cob}(pt,\mathbb{G}_m^{\wedge m}))$ (first by stabilisation, then by passing drom $\mathbb{G}^m$ to $\mathbb{G}^{\wedge m},$ and, finally, by homotopy), gives a well-defined homomorphism of abelian groups

\[\sigma_m: K_m^M(k) \to H^0(C_* \mathbb{Z}F^{cob}(pt,\mathbb{G}_m^{\wedge m})).\]

 Together these homomorphisms form a graded ring homomorphism:

\[\oplus \sigma_m: \bigoplus K_m^M(k) \to \bigoplus H^0(C_* \mathbb{Z}F^{cob}(pt,\mathbb{G}_m^{\wedge m})).\]

\end{proposition}
\begin{proof}
From the construction of the multiplication operation on

 $$\bigoplus H^0(C_* \mathbb{Z}F^{cob}(pt,\mathbb{G}_m^{\wedge m})),$$

 it is obvious that taking the noncommutative monomial  $g_1 \dots g_m$ to the level $0$ correspondence, i.e. map, 

\[\tilde{\sigma}_m(g_1 \dots g_m) : pt \to \mathbb{G}_m^m,\]

gives a homomorphism of (noncommutative) graded rings 

\[ \oplus \tilde{\sigma}_m : \mathbb{Z}\{|k^*|\} \to \bigoplus H^0(C_* \mathbb{Z}F^{cob}(p,\mathbb{G}_m^{\wedge m})) .\]

(multiplication in the ring of noncommutative polynomials is denoted by concatenation, and in $k^*$ --- with the dot $\cdot$).

If we check that the kernel of this homomorphism includes the noncommutative polynomials corresponding to the relations of linearity along each coordinate and the Steinberg relations (these are the relations for the Milnor $K$-theory),  one can see that this homomorphism can be be factored through $\bigoplus K_m^M(k),$ and it follows from the definition that it factors into the homomorphism $\oplus \sigma_m,$ which, in particular, is well-defined. First check the multilinearity:

 \[\tilde{\sigma}_m(g_1 \cdots g_i'\cdots g_m) + \tilde{\sigma}_m(g_1 \cdots g_i''\cdots g_m) =\]

 \[ \tilde{\sigma}_m(g_1 \cdots (g_i' \cdot g_i'')\cdots g_m) + \tilde{\sigma}_m(g_1 \cdots 1 \cdots g_m) .\]

 For this construct a homotopy between the polynomials giving a pair of points $g_i'$ and $g_i''$ and the pair $g_i' \cdot g_i''$ and $1.$ It is given by the polynomial 

\[x^2+(-g_i'-g_i''+t(g_i'+g_i''-g_i' g_i'' -1)) +g_i' g_i''.\]

 Considered as a function on $\mathbb{A}^1 \times \mathbb{A}^1$, it gives a finite over $\mathbb{A}^1$ (along $t$) set $Z$ and an invertible function $x$ on its neighbourhood. By Proposition \ref{reduce_to_value}, the zero- and unit section of this homotopy are equivalent to the left and right sides of the equality.

The Steinberg relations are already true in the Milnor---Witt $K$-theory. In \cite{Milnor-Witt}[8.9] these relations are proven to hold between level $1$ correspondences $[x-a]$ in 

$$H^0(C_* \mathbb{Z}F(pt,\mathbb{G}_m^{\wedge 1})).$$

 By Proposition \ref{cob-gr}, $\Hcob$ gives a homomorphism of graded rings with a unit. Hence the same correspondences hold for the images of these elements in $H^0(C_* \mathbb{Z}F^{cob}(pt,\mathbb{G}_m^{\wedge 1})).$ On the other hand, by Lemma \ref{same} below, the classes of these correspondences are equal to the classes of the maps, or level $0$ correspondences, $\tilde{\sigma}_1(a),$

\end{proof}
\begin{lemma}  \label{same}

Let $[x-a]$ be the correspondence implicitly defined in \cite[Lemma 6.3]{Milnor-Witt}, i.e. the level $1$ framed correspondence $pt \to \mathbb{G}_m$ given by the data 

\[\begin{gathered}
(X=pt,Y=\mathbb{G}_m,Z=pt, id: Z \to X, W=\mathbb{A}^1 \setminus \{0\},\\
W \xrightarrow{i} \mathbb{A}^1, Z \xrightarrow{(a)} W, \phi = (x-a): W \to \mathbb{A}^1, g=id:W \to Y.) 
\end{gathered}\]

Then in $H^0(C_*\mathbb{Z}F(pt,\mathbb{G}_m))$  it has the same class as the map $const_a :pt \to \mathbb{G}_m$
\end{lemma}

\begin{proof}
Applying the stabilisation map to $const_a$, we get a level $1$ correspondence  

\[\begin{gathered} 
\beta= (X=pt,Y=\mathbb{G}_m,Z=pt, id: Z \to X, W=\mathbb{A}^1, W \xrightarrow{id} \mathbb{A}^1, \\
Z \xrightarrow{(0)} W, \phi = (x): W \to \mathbb{A}^1, g=const_a:W \to Y.)
\end{gathered}\]

Construct two homotopies. The first is given by the data

\[ \begin{gathered} 
H_1 = (X=\mathbb{A}^1,Y=\mathbb{G}_m, Z=\mathbb{A}^1, id: Z \to X, W=\mathbb{A}^1 \times \mathbb{A}^1 - \{x-at+a = 0\},\\
W \xrightarrow{i} \mathbb{A}^1 \times \mathbb{A}^1, Z \xrightarrow{(t,at)} W, \phi = t-x-a: W \to \mathbb{A}^1, g= x-at+a: W \to \mathbb{G}_m ).
\end{gathered}\]

(Here $t$  is the homotopy coordinate, both on $X$ and $Z$, and $x$ is the second coordinate on $W$, i.e. the one giving the coordinate function on the fibers $W \to X.$) A computation gives $H_1 \circ i_1 = [x-a].$

$H_1 \circ i_0 = \alpha$ is given by the data 

\[\begin{gathered}
(X=pt,Y=\mathbb{G}_m,Z=pt, id: Z \to X, W=\mathbb{A}^1 -\{x=-a\},\\
W \xrightarrow{i} \mathbb{A}^1, Z \xrightarrow{(0)} W, \phi = (x): W \to \mathbb{A}^1, g=x+a:W \to Y.)
\end{gathered}\]

The second homotopy is
 
\[\begin{gathered}
H_2=(X=\mathbb{A}^1,Y=\mathbb{G}_m, Z=\mathbb{A}^1, id: Z \to X, W=\mathbb{A}^1 \times \mathbb{A}^1 - \{tx = -a\}, \\
W \xrightarrow{i} \mathbb{A}^1 \times \mathbb{A}^1, Z \xrightarrow{(t,0)} W, \phi = x: W \to \mathbb{A}^1, g= tx+a: W \to \mathbb{G}_m ).
\end{gathered}\]

A computation shows that $H_2 \circ i_0 = \beta, H_2 \circ i_1 = \alpha.$  Thus, with the two homotopies, we have connected the two parts of the desired equality.
\end{proof}

$\sigma_m$ is obviously a right inverse to $\rho_m.$  it remains to prove the following:
\begin{proposition} \label{surj}
The map $\sigma_m$ is onto for each $m.$ 
\end{proposition}

\begin{proof}
Denote $can_m:K_m^{MW} \to K_m^M.$ From Lemma \ref{same}, it follows, in the notation of \cite{Milnor-Witt}[8.3], that there is the following commutative square:

\[
\begin{tikzcd}
K_m^{MW} \arrow[r,twoheadrightarrow,"\Psi_m"] \arrow[d,"can_m"] & H^0(C_* \mathbb{Z}F(pt , \mathbb{G}_m^{\wedge m}) \arrow[d,twoheadrightarrow,"\Hcob"]\\
K_m^M \arrow[r,"\sigma_m"] & H^0(C_* \mathbb{Z}F^{cob}(pt,\mathbb{G}_m^{\wedge m}))  
\end{tikzcd}
\]

$\Hcob$ is onto by Corollary \ref{H0-cob}, $\Psi_m$ is onto by the main result of \cite{Milnor-Witt}. Hence $\sigma_m$ is onto.
\end{proof}

First let's construct a map

\[\rho_{m,n,N} : Fr_{n,N}^{cob} (pt,\mathbb{G}_m^{\times m}) \to K_m^M.\]

 Take a correspondence $(Z,W,\phi,g).$ Since $Z$ consists of a finite number of points (call them $z_1, \cdots, z_k$), 

\[\phi^{-1}(s_0(Gr_{n,N}))\]

 is a scheme which in the neighbourhood of each of these $z_i$ is the spectrum of a local Artin ring of length $d_i$. On $W$, and, in particular, on $Z$, $m$  invertible functions $g_1, \dots, g_m$  are given. If $z_i \simeq Spec(F_i),$ then in $K_m^M(F_i)$ there is a symbol $\{g_1|_{z_i},\cdots,g_m|_{z_i}\}$. Taking the sum of norms of these symbols, multiplied by $d_i$ 

\begin{equation} \label{rho.def}
\sum \limits_{i=1}^k d_i \Tr\big|^{F_i}_k \left(\{g_1|_{z_i},\cdots,g_m|_{z_i}\}\right),
\end{equation}
 we get an element in the Milnor $K$-theory of the field $k.$ It can easily be seen that the map defined this way easily translates in a well-defined way to $Fr_{n,N}^{cob} (pt,\mathbb{G}_m^{\wedge m}),$ since any simbol with $1$ in it is equal to $0.$ It is also easily seen to be compatible with the stabilization maps along $n$ and $N,$ and with the extra additivity relations on $\mathbb{Z}F^{cob}(pt,\mathbb{G}_m^{\wedge n}).$ Thus, the maps $\rho_{m,n,N}$ give rise to a homomorphism

\[\mathring{\rho}_m: \mathbb{Z}F^{cob} (pt,\mathbb{G}_m^{\wedge m}) \to K_m^M.\]

\begin{note}\label{basis}
For any $m\in \mathbb N,$ and any $g_1,\dots, g_m\in {\mathbb G}_m(k)$:
$$
\mathring{\rho}_m( \widetilde\sigma_m (g_1\dots g_m))=\{g_1,\dots, g_m\}.
$$
\end{note}

\begin{proposition} \label{rho_m}

The map $\mathring{\rho}_m$ gives rise in a well-defined way to the map
\[\rho_m: H^0(C_* \mathbb{Z}F^{cob}(pt,\mathbb{G}_m^{\wedge m})) \to K_m^M(k).\]

\end{proposition}

To prove this itatement, we will have to study the behaviour of elements of $K$-groups on the curve $Z,$ which is part of a homotopy between to correspondences$pt \to \mathbb{G}_m^M.$ So let's stude the behaviour of $K$-theory on curves with a morphism to $\mathbb{P}^1,$ and prove, in various circumstances, statements that the sum, analogous to the one in Formula \ref{rho.def}, taken for the fiber over $0,$ will be equal to the same sum for the fiber over $1.$ Begin with a smooth curve.

 \begin{lemma} \label{SmoothCurve}
    Let $C$ be a smooth projective curve; $g_1, \cdots ,g_m,f$ be rational functions on $C,$ such that $f=1$ in the zeros and poles of $g_i.$ Let $f$ have zeros 
		
		\[ \text{in points } p^0_1, \cdots, p^0_{r^0}, \text{ with multiplicities }  mul_1^0, \cdots, mul^0_{r^0},\]
		
		and $\frac{1}{f}$ ---
		
		\[ \text{in points } p^{\infty}_1, \cdots, p^{\infty}_{r^{\infty}} \text{,	with multiplicities } mul_1^{\infty}, \cdots, mul^{\infty}_{r^{\infty}}.\]
		
		Then
		
		\[\sum \limits_{i=1}^{r^0} mul^0_i \Tr\big|^{k(p^0_i)}_k \{g_1|_{p^0_i}, \cdots,g_m|_{p^0_i}\} = \sum \limits_{i=1}^{r^{\infty}} mul^{\infty}_i \Tr\big|^{k(p^{\infty}_i)}_k \{g_1|_{p^{\infty}_i}, \cdots,g_m|_{p^{\infty}_i} \} . \]
		
    \end{lemma}
    \begin{proof}
    Consider the symbol
		
		\[\{g_1,\cdots,g_m,f\} \in K_{m+1}(k(Z)).\]
		
		By the Weil reciprocity law \cite{Bass-Tate}[Theorem 5.6], 
		
		\[\sum \limits_{\nu \in \widetilde{Z}} \Tr\big|^{k(\nu)}_k \partial_{\nu} \{g_1,\cdots,g_m,f\} = 0. \]
		
		The sum goes across all closed points (=discrete valuations). $\partial_{\nu}$ is the norm residue map defined in \cite{Bass-Tate}[p. 22, before Proposition 4.4].
    
    Let's calculate the residue $\partial_{\nu} \{g_1,\cdots,g_m,f\}$ in each point $\nu.$ There are two (compatible) possibilities:  Either $f(\nu)=1$, Or $g_1, \cdots g_m \in \mathcal{O}_{\nu}^{\ast}.$
    
    In the first case, our symbol can be written as an algebraic sum of symbols $\{f,h_1, \cdots, h_m\},$ где $h_1,\cdots,h_{m-1} \in \mathcal{O}_{\nu}^{\ast}.$ From \cite{Bass-Tate}[Proposition 4.5 (c)], 
		
		\[\partial_{\nu}\left(\{f,h_1, \cdots, h_m\}\right) =  \nu(h_m)\cdot \{\overline{f},\overline{h_1}, \cdots,\overline{h_{m-1}} \},\]
		
		where the overhead line denotes the common residue, or the value at a point.
    
    In the other case, the same Proposition is applied, without the need for any preparation. In points where $f$ is invertible, the result is $0.$ Only the points$\{ f = 0, \infty\}$ remain. Denoting for each point $p_i^j$ the corresponding valuation as $\nu_i^j,$ we see that
		
		\[\nu_i^0(f) = m_i^0, \nu_i^{\infty}(f) = -m_i^{\infty},\]
		
		and thus, the Weil reciprocity law, together with the explicit formula for residue maps, give the statement of the Lemma.
    \end{proof}
 \begin{lemma} \label{SmoothNilpotents}
 Let $C$  be a one-dimensional projective scheme over $k$ with no embedded points, such that $C^{red}$ is a smooth curve;  let $g_1, \cdots, g_m ,f$ be rational functions on $C,$ such that $f=1$ in all zeros and poles of $g_i.$ Let $f$ have zeros

\[ \text{in points } p^0_1, \cdots, p^0_{r^0}, \text{ with multiplicities }  mul_1^0, \cdots, mul^0_{r^0},\]

 and $\frac{1}{f}$ ---

\[ \text{in points } p^{\infty}_1, \cdots, p^{\infty}_{r^{\infty}} \text{ with multiplicities } m_1^{\infty}, \cdots, m^{\infty}_{r^{\infty}}.\]

 (Here, generalizing the smooth case, multiplicities are taken to be lengths of local Artin rings, which are rings of functions of Artin schemes cut out by the function $f$ or $\frac{1}{f}.$) Then  

\[\sum \limits_{i=1}^{r^0} m^0_i \Tr\big|^{k(p^0_i)}_k \{g_1|_{p^0_i},\cdots,g_m|_{p^0_i}\} = \sum \limits_{i=1}^{r^{\infty}} m^{\infty}_i \Tr\big|^{k(p^{\infty}_i)}_k \{g_1|_{p^{\infty}_i},\cdots,g_m|_{p^{\infty}_i}\} .\]

 \end{lemma}
\begin{proof}
Divide $C$ into connected components $C_k.$ The required equality for $C$ is the sum of equalities for $C_k.$ By the previous Lemma \ref{SmoothCurve}, the equality is true for each $C_k^{red}.$ It suffices to show that the equality for $C_k$ can be acquired from the equality for $C_k^{red}$ by multiplying it by some number $r_k.$ That is shown to be true in the following commutative algebra statement:
\end{proof}
\begin{lemma} \label{loc.free}
Let $C$ be a connected one-dimensional projective scheme over $k$ with no embedded points, such that $C^{red}$ is a smooth curve. There exists a number $r,$ such that for any closed point $p \in C$ and a non-nilpotent function $f \in \mathcal{O}_{C,p},$ the multiplicity of the zero of $f$ in $p$ on $C$ is $r$ times larger than the multiplicity of the zero of $f$ in $p$ on $C^{red}.$
\end{lemma}
\begin{proof}
Let  $\mathcal{I}$ be the nilradical of $\mathcal{O}_C$. Define

\[\mathcal{M}_n = (0 : \mathcal{I}^n) \subset \mathcal{O}_C.\]

$\mathcal{O}_C$ is torsion free as an $\mathcal{O}_C-$module, hence so is its submodule $\mathcal{M}_n$. Let's prove that $\mathcal{M}_n / \mathcal{M}_{n-1} $ is also torsion-free.

Indeed, if $a \cdot m \in \mathcal{M}_{n-1}, m \not \in \mathcal{M}_{n-1}, a \not \in \mathcal{I},$ there exists $s \in \mathcal{I}^{n-1},$ such that $s \cdot m \neq 0.$  But at the same time, $a \cdot s \cdot m = s \cdot a \cdot m = 0,$ making $s \cdot m$ a torsion element, which is a contradiction. 

Thus, the graded $\mathcal{O}_C-$module associated with the filtration $\mathcal{M}_n$ is torsion-free. Since $C$ has no embedded points, it also means that it is torsion-free as an $\mathcal{O}_C / \mathcal{I}-$module. Since the curve $C^{red}$ is smooth, this means that the module is locally free. Let $r$ be its rank, and $r_n$ the rank of its $n$th graded component.

 Let $p \in C$ be a closed point, $f \in \mathcal{O}_{C,p}.$ Localising at $p$ gives a filtration $\mathcal{M}_{n,p}.$ On its intermediate quotients $f$ is a nonzerodivisor, hence $(f) \cap \mathcal{M}_{n,p} = f \cdot \mathcal{M}_{n,p}.$ Thus the associated filtration of the module $\mathcal{O}_{C,p}/(f)$ is equal to 

\[\mathcal{M}_{n,p}/ \big ((f) \cap \mathcal{M}_{n,p} \big ) = \mathcal{M}_{n,p}/ (f \cdot \mathcal{M}_{n,p}).\]

 Its intermediate quotients are

\[ \frac{ \mathcal{M}_{n,p} / \mathcal{M}_{n-1,p}} {f \cdot ( \mathcal{M}_{n,p} / \mathcal{M}_{n-1,p}) } .\]

The module

\[\frac{ \mathcal{M}_{n,p} / \mathcal{M}_{n-1,p}} {f \cdot ( \mathcal{M}_{n,p} / \mathcal{M}_{n-1,p}) }\] 

is isomorphic to $\left( \mathcal{O}_{C,p}/(\mathcal{I}+(f)) \right)^{\oplus r_n},$ hence its length is equal to $d \cdot r_n,$ where $d$ is the length of $\mathcal{O}_{C,p}/(\mathcal{I}+(f)),$ which is the multiplicity of the zero of $f$ in $p$ on the curve $C^{red}.$  Since these spaces are the intermediate quotients of a filtration of the Artin ring $\mathcal{O}_{C,p}/(f)$, its length is equal to the sum of their dimensions  $\sum d \cdot r_i = d \cdot r.$ Which means that the multiplicity of the zero of $f$ in $p$ on $C$ is equal to $d \cdot r.$
\end{proof}

Let's proceed to the proof of the Proposition
\begin{proof}

It is sufficient to show that for any homotopy

\[h:(Z,W,\phi,g) : \mathbb{A}^1 \to \mathbb{G}_m, \mathring{\rho}_m(h \circ j_0) = \mathring{\rho}_m(h \circ j_1),\]

 where $j_0, j_1 $ are the embeddings of  $0$  $1$ into $\mathbb{A}^1_k.$ Denote the structural map $Z \to \mathbb{A}^1$ by $\pi.$ 

Note that, since $Z$ is regularly embedded into a smooth variety, it has no embedded points.

 Preserving the same notation, substitute $Z_{Sch}$ for its projective closure. On $Z_{Sch}$ there are rational functions $g_i$ and $f=\frac{\pi^{\ast} t}{\pi^{\ast} t - 1},$ with the property that $g_1,\cdots,g_m$ are invertible regular functions on the complement of $\{f=1\}.$

 The normalisation of $Z$ is a smooth projective curve $\widetilde{Z}$. The normalisation of a curve can be acquired by sequential blowup of points. Let $\widetilde{Z_{Sch}} \supset \widetilde{Z}$ be the result of blowing up $Z_{Sch}$ in the same sequence of points. 

Each blowup results in the insertion of an effective Cartier divisor (see \cite[Definition 22.2.0.1]{FOAG}), hence on $\widetilde{Z_{Sch}}$ there are no embedded points. By \cite[Lemma 22.2.6]{FOAG}, the proper transform of $Z$ in $\widetilde{Z_{Sch}}$ is equal to $\widetilde{Z}.$ Since each blowup results in an insertion of an effective Cartier divisor, $\widetilde{Z_{Sch}}$ has no new irreducible components in comparison with $Z_{Sch}.$ Since $\widetilde{Z_{Sch}}^{red}$ has the same number of components, the reduced part $\widetilde{Z_{Sch}}^{red} = \widetilde{Z}$ is a smooth curve. Thus, $\widetilde{Z_{Sch}}$ satisfies the conditions of the previous Lemma \ref{SmoothNilpotents}.

Denote the zeros and poles of $f$ on $Z_{Sch}$ by

\[p_1^0, \cdots p_{r^0}^0 ; p_1^{\infty} \cdots p_{r^{\infty}}^{\infty}, \text{ with multiplicities } mul_1^0, \cdots mul_{r^0}^0; mul_1^{\infty}, \cdots mul_{r^{\infty}}^{\infty}.\]

On $\widetilde{Z_{Sch}}$ --- by

\[\tilde{p}_1^0, \cdots \tilde{p}_{\tilde{r}^0}^0 ; \tilde{p}_1^{\infty} \cdots \tilde{p}_{\tilde{r}^{\infty}}^{\infty}, \text{ with multiplicities } \widetilde{mul}_1^0, \cdots \widetilde{mul}_{\tilde{r}^0}^0; \widetilde{mul}_1^{\infty}, \cdots \widetilde{mul}_{\tilde{r}^{\infty}}^{\infty}.\]

We have three equaltites in the following chain, completing which will prove the Proposition:
\[
\begin{tikzcd}
\mathring{\rho}_m(h \circ j_0) \arrow[d,equal] & \mathring{\rho}_m(h \circ j_1) \arrow[d,equal]  \\
 \sum \limits_{i=1}^{r^0} mul^0_i \Tr\big|^{k(p^0_i)}_k \{g_1|_{p^0_i}, \cdots,g_m|_{p^0_i}\} \arrow[d,equal,dotted] 
& \sum \limits_{i=1}^{r^{\infty}} mul^{\infty}_i \Tr\big|^{k(p^{\infty}_i)}_k \{g_1|_{p^{\infty}_i}, \cdots,g_m|_{p^{\infty}_i}\} \arrow[d,equal,dotted]\\
 \sum \limits_{i=1}^{\tilde{r}^0} \widetilde{mul}^0_i \Tr\big|^{k(\tilde{p}^0_i)}_k \{g_1|_{\tilde{p}^0_i}, \cdots,g_m|_{\tilde{p}^0_i}\} \arrow[r,equal]
& \sum \limits_{i=1}^{\tilde{r}^{\infty}} \widetilde{mul}^{\infty}_i \Tr\big|^{k(\tilde{p}^{\infty}_i)}_k \{g_1|_{\tilde{p}^{\infty}_i}, \cdots,g_m|_{\tilde{p}^{\infty}_i}\}
\end{tikzcd}
\]

The two dotted equalities are analogous to each other, so we'll only prove the leftmost one.
It states that the sum for $Z$ is equal to the analogous sum for $\widetilde{Z_{Sch}}.$ This is not obvious, since, for example, a point with nilpotents on $Z_{Sch}$ can turn into several points on $\widetilde{Z_{Sch}},$ with other residue fields and nilpotents. Moreover, by construction of $\widetilde{Z_{Sch}}$, the nilpotents over those points "collapse" under the map to $Z_{Sch}.$ 

To prove this, pull back to the scheme $ V = Spec k[[t]]$. Since outside of the blow-up points the scheme does not change, $\widetilde{-}_V : \widetilde{Z_{Sch}}_V \to Z_{Sch,V}$ is an isomorphism over the complement $V-v$ of the closed point. 

If $Y$ is a connected component of $Z_{Sch,V}$ (a local scheme), and $\widetilde{Y}_k$ are all the connected components of its preimage in $\widetilde{Z_{Sch}},$ then $Y$ and $\coprod \widetilde{Y}_k$ are both finite and flat (since their structure rings are torsion-free $k[[t]]-$modules) over $V.$ They alse have the same degree over $V$, since they are isomorphic over $V-v$. The required equality is given by adding all the equalities from the following Lemma across all connected components $Y$ (This divides the sum for $\widetilde{Z_{Sch}}$ into subsums) :
\begin{lemma}
Let $S, \widetilde{S}$ be two finite schemes of the same degree over the field $k$, such that $S$ is local, and there is a morphism $\widetilde{-}: \widetilde{S} \to S$. 
Denote the points of $S$ and $\widetilde{S}$ as $s$ and $\{\tilde{s_i}\}.$ Let $S_i$ be the component of $\widetilde{S}$ containing $s_i.$ Denote the induced morphism $\pi_i: s_i \to s.$ Let $a \in K_m^M(k(s))$.  Denote the length of $S$ as $d$, and the length of $S_i$ as $d_i.$ Then 

\[d \cdot \Tr\big|^{k(s)}_k (a) =\sum \limits_k  d_i \Tr\big|^{k(s_i)}_k ((\pi_i)^* a).\]
\end{lemma}
\begin{proof}
Denote the degrees of field extensions 

\[p=[k(s) : k], p_i = [k(s_i) : k].\]

Then

\[\Tr\big|^{k(s_i)}_k = \Tr\big|^{k(s)}_k \circ \Tr\big|^{k(s_i)}_{k(s)}.\] 

Hence 

\begin{multline*}
\sum \limits_i  d_i \cdot \Tr\big|^{k(s_i)}_k ((\pi_i)^* a) = 
\sum \limits_i  d_i \cdot \Tr\big|^{k(s)}_k \left( \Tr\big|^{k(s_i)}_{k(s)} ((\pi_i)^* a) \right) =\\
\Tr\big|^{k(s)}_k \left( \sum \limits_i  d_i \cdot \Tr\big|^{k(s_i)}_{k(s)} ((\pi_i)^* a) \right)
\end{multline*}

For the Milnor $K$-theory, the composition of extension of scalars and the norm map is multiplication by the degree of the field extension, 

\[\Tr\big|^{k(s_i)}_{k(s)} ((\pi_i)^* a) = \frac{p_i}{p} a.\]

 Hence

\[
\Tr\big|^{k(s)}_k \left( \sum \limits_i d_i \cdot \Tr\big|^{k(s_i)}_{k(s)} ((\pi_i)^* a) \right)=
\Tr\big|^{k(s)}_k \left( \sum \limits_i d_i  \frac{p_i}{p} a \right) = 
\frac{1}{p} \cdot \sum \limits_i (d_i \cdot p_i)  \cdot \Tr\big|^{k(s)}_k a 
\]

Since $S$ and $\widetilde{S}$ have the same degree, $\sum \limits_i (d_i \cdot p_i) = d \cdot p.$ Hence 

\[
\frac{1}{p} \cdot \sum \limits_i (d_i \cdot p_i)  \cdot \Tr\big|^{k(s)}_k a  =
\frac{1}{p} \cdot d \cdot p \cdot \Tr\big|^{k(s)}_k a =
d \cdot \Tr\big|^{k(s)}_k a.
\]

\end{proof}

\end{proof}

Thus $\rho_m$ and $\sigma_m$ are mutually inverse isomorphisms, completing the proof of Theorem \ref{main}


\begin{thebibliography}{EGAIII}
\bibitem[BT]{Bass-Tate}
\emph{H. Bass and J. Tate} The Milnor ring of a global Field, Algebraic K-theory II: "Classical" Algebraic K-Theory and Connections with Arithmetic (Proc. Conf., Seattle, Wash., Battelle
Memorial Inst., 1972),
Springer, Berlin, 1973, pp. 349–446. Lecture Notes in Math., Vol. 342. MR0442061 (56 \#449)
\bibitem[Nes]{Milnor-Witt}
\emph{A. Neshitov},2016, FRAMED CORRESPONDENCES AND THE MILNOR–WITT $K$ -THEORY. Journal of the Institute of Mathematics of Jussieu, 1-30. doi:10.1017/S1474748016000190

\bibitem[GP]{GP}
 Garkusha~G.,  Panin~I., {\it Framed motives of algebraic varieties} ({\it after V. Voevodsky}), {\tt	arXiv:1409.4372 [math.KT]}, 2014.


\bibitem[Voev]{Voev}
 Voevodsky~V., {\it Notes on framed correspondences}, unpublished, 2001.          Available at {\tt math.ias.edu/vladimir/files/framed.pdf}.

\bibitem[SV]{SV}  Suslin~A.,  Voevodsky~V., {\it Bloch--Kato conjecture and motivic cohomology with finite coefficients},  The Arithmetic and Geometry of Algebraic Cycles,
 (Banff, AB, 1998), NATO Sci. Ser. C Math. Phys. Sci., vol. 548,
 Kluwer Acad. Publ., Dordrecht, 2000, pp. 117--189.

\bibitem[FOAG]{FOAG}
\emph{R. Vakil},The Rising Sea: Foundations of Algebraic Geometry \url{http://math.stanford.edu/~vakil/216blog/FOAGnov1817public.pdf}

\bibitem[GarNesh]{GarNesh}
 Garkusha~G.,  Neshitov~A.,  {\it Fibrant resolutions for motivic Thom spectra}, {\tt arXiv:1804.07621 [math.AG]}. 20 Apr 2018.


\end{thebibliography}
\end{document}